\RequirePackage{fix-cm}
\documentclass[envcountsect]{svjour3} 
\smartqed 

\usepackage{amsmath, mathrsfs, amscd, amsfonts, amssymb, graphicx, color} 
\usepackage[bookmarksnumbered, colorlinks, plainpages]{hyperref}
\hypersetup{colorlinks=true,linkcolor=blue, anchorcolor=blue, citecolor=blue, urlcolor=blue, filecolor=blue, pdftoolbar=true}

\numberwithin{equation}{section}

\begin{document}

\title{Bounds for zeros of a polynomial using numerical radius  of Hilbertian  space operators}
\author{Pintu Bhunia \and
 Santanu Bag \and Kallol Paul 
}


\institute{Pintu Bhunia \at
Department of Mathematics, Jadavpur University, Kolkata 700032, West Bengal, India\\
 \email{pintubhunia5206@gmail.com}
 \and
 Santanu Bag \at
Department of Mathematics, Vivekananda College For Women, Barisha, Kolkata 700008, India\\
 \email{santanumath84@gmail.com} 
\and
Kallol Paul \at
Department of Mathematics, Jadavpur University, Kolkata 700032, West Bengal, India\\
\email{kalloldada@gmail.com}
}

\date{Received: date / Accepted: date}

\maketitle

\begin{abstract}
We obtain bounds for the numerical radius of $2 \times 2$ operator matrices which improve on the existing bounds. We also show that the inequalities obtained here generalize the existing ones. As an application of the results obtained here we estimate the bounds for the zeros of a monic polynomial and  illustrate with numerical examples  that the bounds are better than the existing ones.
\end{abstract}

\keywords{Numerical radius \and Operator matrix \and Zeros of polynomial}

\subclass{ 47A12 \and 15A60 \and 26C10}
\section{Introduction}
\label{intro}
We consider a monic polynomial $p(z)=z^n+a_{n-1}z^{n-1}+\ldots+a_1z+a_0$ of degree $n$,  with complex coefficients $a_0, a_1, \ldots, a_{n-1}.$ When $n$ varies from $1$ to $4$, we can exactly compute the zeros of the polynomial $p(z)$. But for $n \geq 5$, there is no general method to compute the zeros of the polynomial $p(z)$ and for this reason the estimation of bounds for the zeros of polynomials becomes more interesting. Several classical bounds for the zeros of a given monic polynomial have been obtained by different mathematicians over the years using different approaches. To mention a few of them are Cauchy \cite{CM}, Fujii and Kubo \cite{FK1}, Alpin et al. \cite{YML}, Kittaneh \cite{K}, Linden \cite{L}. One of the important technique to obtain bounds for the zeros of the polynomial $p(z)$ is to obtain bounds for the numerical radius of the Frobenius companion matrix $C(p)$ of $p(z),$ where \begin{eqnarray*}
  C(p)&=&\left(\begin{array}{ccccccc}
    -a_{n-1}&-a_{n-2}&.&.&.&-a_1&-a_0 \\
    1&0&.&.&.&0&0\\
		0&1&.&.&.&0&0\\
    .& & & & & & \\
    .& & & & & &\\
    .& & & & & &\\
    0&0&.&.&.&1&0
    \end{array}\right).
		\end{eqnarray*}
Using the numerical radius inequalities of the Frobenius companion matrix of a given monic polynomial, Abu-Omar and Kittaneh \cite{AF}, M. Al-Dolat et al. \cite{MKMF}, Bhunia et al. \cite{BBP} obtained various bounds for the zeros of that polynomial. We here obtain bounds for the zeros of the polynomial $p(z)$ and give examples to  show that they are better than the existing ones. Before we proceed further we first discuss the necessary notation and terminology.\\
 
\noindent Let $\mathcal{H}$ be a complex Hilbert space with inner product $\langle .,. \rangle$ and let $\mathcal{B}(\mathcal{H})$ denote the $C^{*}$-algebra of all bounded linear operators on $H$.   Also let  $B(\mathcal{H}_1, \mathcal{H}_2)$ be the set of all bounded linear operators from $\mathcal{H}_1$ to $\mathcal{H}_2$. Here $\mathcal{H}_1$ and $ \mathcal{H}_2$ are two complex Hilbert spaces. If $\mathcal{H}_1=\mathcal{H}_2=\mathcal{H}$ then we write $\mathcal{B}(\mathcal{H}_1, \mathcal{H}_2)=\mathcal{B}(\mathcal{H})$.  The numerical range of $T\in \mathcal{B}(\mathcal{H}),$ denoted as $ W(T)$,  is  defined as $$ W(T)=\{ \langle Tx,x \rangle: x \in H,~ \|x\|=1\},$$ where $\|.\|$ denotes the norm  on $\mathcal{H}$ induced by the inner product $\langle .,. \rangle$. For any bounded linear operator $T$, the numerical radius, denoted as $w(T)$ and the Crawford number, denoted as $m(T)$ are defined respectively, as
\begin{eqnarray*}
 w(T)&=&\sup\{ |\lambda|: \lambda \in W(T)\}
\end{eqnarray*}
and 
\begin{eqnarray*}
 m(T)&=&\inf\{ |\lambda|: \lambda \in W(T)\}.
\end{eqnarray*}
Let $\rho(T)$ and $\|T\|$ denote the spectral radius and the operator norm of $T$, respectively. It is easy to see that $\rho(T)\leq w(T)\leq \|T\|.$ 
Also, $w(.)$ is a norm on $\mathcal{B}(\mathcal{H})$ that satisfies the inequality
\begin{eqnarray}\label{number0}  
 \frac{1}{2}\|T\|&\leq& w(T)\leq \|T\|.
\end{eqnarray}
Note that this implies that $w(.)$ is equivalent to the operator norm. Various numerical radius inequalities improving (\ref{number0}) have been studied in \cite{BBP}, \cite{OFK}, \cite{KMR}, \cite{PB2}, \cite{PB1},  \cite{KS}, \cite{TY}.  Kittaneh \cite{FK3} improved on the inequality (\ref{number0}) to prove that 
\begin{eqnarray}\label{number1}
\frac{1}{4}\left\|T^{*}T+TT^{*}\right\|\leq w^2(T)\leq \frac{1}{2}\left\|T^{*}T+TT^{*}\right\|,  
\end{eqnarray}
which was further improved in \cite{BBP1} as 
\begin{eqnarray}\label{number2}
	\frac{1}{4}m\left(\left(\mbox{Re}\left(T^2\right)\right)^2\right)+\frac{1}{16}  \left\|T^{*}T+TT^{*} \right\|^2  \leq w^4(T)
\end{eqnarray}
\begin{eqnarray}\label{number3}
	\mbox{and}~~w^4(T) \leq  \frac{1}{2}w^2(T^2)+\frac{1}{8} \left\|T^{*}T+TT^{*} \right\|^2.
	\end{eqnarray}

\noindent We know that the zeros of the polynomial $p(z)$ are  exactly the eigenvalues of $C(p)$. Therefore,  if $\lambda$ is a zero of the polynomial  $p(z),$ then $|\lambda|\leq w(C(p))$. Thus if we can obtain better bounds for the numerical radius of an operator then we can definitely improve on the bounds for the zeros of a given polynomial.\\

\noindent In  this paper, we establish the numerical radius inequalities for $2 \times 2$ operator matrices which generalize and improve on the inequalities (\ref{number1}), (\ref{number2}) and (\ref{number3}) .
 Using these inequalities, we obtain upper bounds for the numerical radius of the product of two operators. We  develop upper bounds for the norm of the product of two positive operators and that of the sum of two operators. Also we show with numerical examples that these bounds improve on the bound obtained by Shebrawi \cite{KS}. Finally, as an application of these numerical radius inequalities of $2 \times 2$ operator matrices,  we estimate bounds for the zeros of a monic polynomial. We show with numerical examples that these estimations improve on the existing estimations.

\section{On the bounds of numerical radius}\label{sec1}

We first introduce the notations  $H_{\theta}$ and $K_{\theta}$.
For $T \in B(\mathbb{H})$ and $\theta \in \mathbb{R},$  let $H_{\theta}=\textit{Re}(e^{i\theta}T)$ and $K_{\theta}=\textit{Im}(e^{i\theta}T)$.  
The following  lemma  will be  used  repeatedly to reach our goal in this present article.

\begin{lemma}[\cite{TY}] \label{lemma:lemma1}
	Let $T \in \mathcal{B}(\mathcal{H})$, then \[w(T)=\sup_{\theta \in \mathbb{R}}\|H_{\theta} \| ~~\mbox{and}~~ w(T)=\sup_{\theta \in \mathbb{R}}\| K_{\theta} \|.\]
\end {lemma}		
	
Now we are in a position to prove the following inequalities for the numerical radius of  $2\times 2$ operator matrices which generalize the existing inequalities. 

\begin{theorem}\label{theorem:upperlower1}
Let $X \in \mathcal{B}(\mathcal{H}_2,\mathcal{H}_1), Y \in \mathcal{B}(\mathcal{H}_1,\mathcal{H}_2)$. Then
\begin{eqnarray*}
w^{2}\left(\begin{array}{cc}
    0&X \\
    Y&0
 \end{array}\right)  &\geq& \frac{1}{4}\max \big\{\| XX^{*}+Y^{*}Y\|, \| X^{*}X+YY^{*}\|\big\}, \\ 
w^{2}\left(\begin{array}{cc}
    0&X \\
    Y&0
 \end{array}\right) 
&\leq& \frac{1}{2}\max \big\{\| XX^{*}+Y^{*}Y\|, \| X^{*}X+YY^{*}\|\big\}. 
\end{eqnarray*}
\end{theorem}

\begin{proof} Let $T=\left(\begin{array}{cc}
    0&X \\
    Y&0
 \end{array}\right)$ and $H_{\theta}=\textit{Re}(e^{i\theta}T)$, $K_{\theta}=\textit{Im}(e^{i\theta}T).$ An easy calculation gives
\[H^2_{\theta}+K^2_{\theta}=\frac{1}{2}\left(\begin{array}{cc}
    A&0 \\
    0&B
 \end{array}\right),\] where $A=XX^{*}+Y^{*}Y$, $B=X^{*}X+YY^{*}$.
Therefore, using Lemma \ref{lemma:lemma1}, we get $$\frac{1}{2}\Big\|\left(\begin{array}{cc}
    A&0 \\
    0&B
 \end{array}\right)\Big\|=\|H^2_{\theta}+K^2_{\theta}\|\leq \|H_{\theta}\|^2+\|K_{\theta}\|^2\leq 2w^2(T).$$ This shows that 
$$\frac{1}{2}\max \big\{\| A\|, \| B\|\big\}\leq 2w^2(T).$$ 
This completes the proof of the first inequality of the theorem.\\
Again, from $H^2_{\theta}+K^2_{\theta}=\frac{1}{2}\left(\begin{array}{cc}
   A&0 \\
    0&B
 \end{array}\right) $,~~\mbox{we have}  $$H^2_{\theta}-\frac{1}{2}\left(\begin{array}{cc}
    A&0 \\
    0&B
 \end{array}\right)=-K^2_{\theta}\leq 0.$$
Therefore, $$ H^2_{\theta}\leq\frac{1}{2}\left(\begin{array}{cc}
    A&0 \\
    0&B
 \end{array}\right)$$ and so,  
$$\|H_{\theta}\|^2\leq\frac{1}{2}\Big\|\left(\begin{array}{cc}
    A&0 \\
    0&B
 \end{array}\right)\Big\|=\frac{1}{2}\max \big\{\| A\|, \|B\|\big\}.$$
Taking supremum over $\theta\in \mathbb{R}$ and then using Lemma \ref{lemma:lemma1}, we get
$$w^{2}(T) \leq \frac{1}{2}\max \big\{\| A\|, \|B\|\big\}.$$ This completes the proof of the second inequality of the theorem. 
\end{proof}

\begin{remark}\label{remark1}
(i) It is well known that if $\mathcal{H}_1=\mathcal{H}_2$ and $X=Y$ then $w\left(\begin{array}{cc}
    0&X \\
    Y&0
 \end{array}\right)=w(X)$. Using this result and taking $X=Y=T$ (say) in Theorem \ref{theorem:upperlower1}, we get 
 \[ \frac{1}{4} \left\|T^{*}T+TT^{*}\right\| \leq w^2(T) \leq \frac{1}{2}\left\|T^{*}T+TT^{*}\right\|,  \]
 which is the inequality (\ref{number1}).\\
 (ii) In \cite[Th. 2.6]{BPN}, we obtained analogous bounds for  $A$-numerical radius in the setting of  semi-Hilbertian space operators.

\end{remark}

Using  Theorem \ref{theorem:upperlower1} we now obtain an upper bound for the norm of the sum of two bounded linear operators.
\begin{theorem} \label{theorem1}
Let $X, Y \in \mathcal{B}(\mathcal{H}).$ Then
\[\left\|X+Y\right\|^2\leq 2\max \big\{\| XX^{*}+YY^{*}\|, \| X^{*}X+Y^{*}Y\|\big\}.\]
\end{theorem}

\begin{proof}
Let $T=\left(\begin{array}{cc}
    0&X \\
    Y{^*}&0
 \end{array}\right).$ Then it follows from $w(T)\geq \|\textit{Re}(T)\|$ that  $$\|X+Y\|^2 \leq 4w^2(T).$$ Thus using Theorem \ref{theorem:upperlower1}, we get
$$\left\|X+Y\right \|^2\leq 2\max \big\{\| XX^{*}+YY^{*}\|, \| X^{*}X+Y^{*}Y\|\big\},$$
as required.
\end{proof}

In the next theorem we obtain a norm inequality for positive operators.

\begin{theorem} \label{therem3}
Let $X, Y$ be positive operators in $\mathcal{B}(\mathcal{H}).$ Then
\[\left\|X^{\frac{1}{2}}Y^{\frac{1}{2}}\right \|^2\leq \frac{1}{2}\left\| X^2+Y^2\right\|.\]
\end{theorem}
\begin{proof}
We have 
$\rho(XY)=\left\|X^{\frac{1}{2}}Y^{\frac{1}{2}}\right\|^2$. Also we have the operator inequality 
$$2\textit{Re}(e^{i\theta}AB) \leq AA^*+B^*B, ~ \mbox{for all} ~\theta \in \mathbb{R}$$ 
 and so $w(AB) \leq \frac{1}{2}\| AA^{*}+B^{*}B\|$, for all $A,B\in \mathcal{B}(\mathcal{H}).$ 
 Thus, 
 $$\rho(XY)=\left\|X^{\frac{1}{2}}Y^{\frac{1}{2}}\right\|^2 \leq w(XY) \leq \frac{1}{2} \left\| X^2+Y^2\right\|.$$
 This completes the proof of the theorem.
\end{proof}

In the following theorem we obtain another numerical radius inequalities for $2\times2$ operator matrices.

\begin{theorem}\label{theorem:upperlower2}
Let $X \in \mathcal{B}(\mathcal{H}_2,\mathcal{H}_1), Y \in \mathcal{B}(\mathcal{H}_1,\mathcal{H}_2)$. Then\\
$w^{4}\left(\begin{array}{cc}
	0&X \\
	Y&0
\end{array}\right)  \geq  \frac{1}{16}\max \big\{\| A_0\|, \| B_0\|\big\} ~\mbox{and}$\\ 
$w^{4}\left(\begin{array}{cc}
	0&X \\
	Y&0
\end{array}\right) \leq \frac{1}{8}\max \Big\{\| XX^{*}+Y^{*}Y\|^2+4w^2(XY), \| X^{*}X+YY^{*}\|^2  + 4w^2(YX)\Big\},$
where $A_0=(XX^{*}+Y^{*}Y)^2 + 4(\textit{Re}(XY))^2,B_0=(X^{*}X+YY^{*})^2 + 4(\textit{Re}(YX))^2.$

\end{theorem}

\begin{proof} Let $T=\left(\begin{array}{cc}
    0&X \\
    Y&0
 \end{array}\right)$ and $H_{\theta}=\textit{Re}(e^{i\theta}T)$, $K_{\theta}=\textit{Im}(e^{i\theta}T).$ An easy calculation gives
\[H^4_{\theta} + K^4_{\theta}=\frac{1}{8}\left(\begin{array}{cc}
    A&0 \\
    0&B
 \end{array}\right),\]
where $A=(XX^{*}+Y^{*}Y)^2 + ~4(\textit{Re}(e^{2i\theta}XY))^2, B=(X^{*}X+YY^{*})^2 + ~4(\textit{Re}(e^{2i\theta}YX))^2.$
Therefore, using Lemma \ref{lemma:lemma1}, we get
$$\frac{1}{8} \Big \|\left(\begin{array}{cc}
    A&0 \\
    0&B
 \end{array}\right)\Big\|=\|H^4_{\theta}+K^4_{\theta}\|\leq \|H_{\theta}\|^4+\|K_{\theta}\|^4\leq 2w^4(T).$$ This shows that 
$$\frac{1}{8}\max \big\{\| A\|, \| B\|\big\}\leq 2w^4(T).$$ This holds for all $\theta \in \mathbb{R}$, so taking $\theta=0$, we get
$$\frac{1}{8}\max \big\{\| A_0\|, \| B_0\|\big\} \leq 2w^{4}(T).$$ 
 This completes the proof of the first inequality of the theorem.

\noindent Again, from $H^4_{\theta}+K^4_{\theta}=\frac{1}{8}\left(\begin{array}{cc}
    A&0 \\
    0&B
 \end{array}\right)$, we have
$$H^4_{\theta}-\frac{1}{8}\left(\begin{array}{cc}
    A&0 \\
    0&B
 \end{array}\right)=-K^4_{\theta}\leq 0.$$
Thus, $$ H^4_{\theta}\leq\frac{1}{8}\left(\begin{array}{cc}
    A&0 \\
    0&B
 \end{array}\right),$$ and so,
$$\|H_{\theta}\|^4\leq\frac{1}{8}\Big\|\left(\begin{array}{cc}
    A&0 \\
    0&B
 \end{array}\right)\Big\|=\frac{1}{8}\max \big\{\| A\|, \| B\|\big\}.$$
Hence,
$$\|H_{\theta}\|^4\leq\frac{1}{8}\max  \big\{\| XX^{*}+Y^{*}Y\|^2+4w^2(XY), \| X^{*}X+YY^{*}\|^2+4w^2(YX)\big\}.$$ 
Taking supremum over $\theta\in \mathbb{R}$ and using Lemma \ref{lemma:lemma1}, we get
$$w^{4}(T) \leq \frac{1}{8}\max  \big\{\| XX^{*}+Y^{*}Y\|^2+4w^2(XY), \| X^{*}X+YY^{*}\|^2+4w^2(YX)\big\}.$$  This completes the proof of the second inequality of the theorem. 
\end{proof}
\begin{remark}
	In \cite[Th. 2.9]{BPN}, we obtained analogous bounds for  $A$-numerical radius in the setting of  semi-Hilbertian space operators. 
	\end{remark}

Considering $\mathcal{H}_1=\mathcal{H}_2= \mathcal{H}$ and $X=Y=T$ (say) in Theorem \ref{theorem:upperlower2}, we get the following corollary.

\begin{corollary}\label{cor4}
Let $T\in \mathcal{B}(\mathcal{H}).$ Then
\begin{eqnarray*}
\frac{1}{16}\left\|\left(TT^{*}+T^{*}T\right)^2+4\left(\textit{Re}\left(T^2\right)\right)^2\right\| \leq  w^4(T) \leq  \frac{1}{8}\left\| TT^{*}+T^{*}T\right\|^2+\frac{1}{2}w^2(T^2). 
\end{eqnarray*}
\end{corollary}

\begin{remark}
It is well known that $2w(T^2)\leq \| TT^{*}+T^{*}T\|.$ Also it is easy to show that $\left \|\left(TT^{*}+T^{*}T\right)^2+4\left(\textit{Re}\left(T^2\right)\right)^2\right \|\geq \left \|TT^{*}+T^{*}T\right \|^2+4 m\left(\left(\mbox{Re}\left(T^2\right)\right)^2\right).$ Thus the inequality obtained by us in Corollary \ref{cor4} improves on the inequality (\ref{number1}), it also improves on the inequality (\ref{number2}).\\

\end{remark}

Next, by using Theorem \ref{theorem:upperlower2}, we prove another norm inequality for the sum of two bounded linear operators.

\begin{theorem} \label{theorem5}
	Let $X, Y \in \mathcal{B}(\mathcal{H}).$ Then
	\[\|X+Y\|^4\leq 2 \max  \big\{\| XX^{*}+YY^{*}\|^2+4w^2(XY^{*}), \| X^{*}X+Y^{*}Y\|^2+4w^2(Y^{*}X)\big\}.\]
\end{theorem}
\begin{proof}
Let $T=\left(\begin{array}{cc}
    0&X \\
    Y{^*}&0
 \end{array}\right).$ Then it follows from $w(T)\geq \|\textit{Re}(T)\|$ that $$\|X+Y\|^4\leq 16 w^4(T).$$ Therefore, using Theorem \ref{theorem:upperlower2}, we get
\[\|X+Y\|^4\leq 2 \max  \big\{\| XX^{*}+YY^{*}\|^2+4w^2(XY^{*}), \| X^{*}X+Y^{*}Y\|^2+4w^2(Y^{*}X)\big\},\] as required.
\end{proof}

\begin{remark}
	Abu-Omar and Kittaneh \cite{AF} proved that for any operators $X,Y \in \mathcal{B}(\mathcal{H}),$
	\begin{eqnarray}\label{number4}
	\|X+Y\|&\leq& \max \{\|X\|, \|Y\|\}+\max\big\{\||X|^{\frac{1}{2}}|Y|^{\frac{1}{2}}\|,\||X^{*}|^{\frac{1}{2}}|Y^{*}|^{\frac{1}{2}}\|\big\}
	\end{eqnarray}
	and Shebrawi \cite{KS} improved on the inequality (\ref{number4}) to prove that
	\begin{eqnarray}\label{number5}
	\|X+Y\|&\leq& \max \{\|X\|, \|Y\|\}+\frac{1}{2}\big(\||X|^{\frac{1}{2}}|Y|^{\frac{1}{2}}\|+\||X^{*}|^{\frac{1}{2}}|Y^{*}|^{\frac{1}{2}}\|\big).
	\end{eqnarray}
We provide a numerical example to claim that the  bounds for the norm of the sum of two operators obtained in Theorem \ref{theorem1} and Theorem \ref{theorem5}
improve on the bound mentioned in (\ref{number5}).
Let $X=\left(\begin{array}{cc}
    2&0 \\
    0&0
	\end{array}\right)$ and $Y=\left(\begin{array}{cc}
    3&0 \\
    0&0
	\end{array}\right)$. Then the inequality (\ref{number5}) gives $\|X+Y\| \leq 3+\sqrt{6},$ whereas that  in Theorem \ref{theorem1} and Theorem \ref{theorem5} gives respectively $\|X+Y\| \leq \sqrt{26}$ and $\|X+Y\| \leq \sqrt{\sqrt{626}}.$ Thus our claim is established.
\end{remark}

\vspace{.3cm}

We conclude this section with two theorems related to upper and lower bounds for the numerical radius of general $2 \times 2$ operator matrices. To do so we need the following lemma, the proof of which is in  \cite[p. 107]{B}. 
\begin {lemma}\label{lemma:lemma3}
Let $X, Y, Z, W \in \mathcal{B}(\mathcal{H})$. Then

\[w\left(\begin{array}{cc}
    X&Y \\
    Z&W
	\end{array}\right)\geq w\left(\begin{array}{cc}
    0&Y \\
    Z&0
	\end{array}\right),
	w\left(\begin{array}{cc}
    X&Y \\
    Z&W
	\end{array}\right)\geq w\left(\begin{array}{cc}
    X&0 \\
    0&W
	\end{array}\right).\]
\end{lemma}
    
\begin{theorem}\label{theorem:estimate1}
Let $X, Y, Z, W \in \mathcal{B}(\mathcal{H})$. Then
\[w\left(\begin{array}{cc}
    X&Y \\
    Z&W
	\end{array}\right)\leq \max \bigg\{ w(X), w(W) \bigg\} + \bigg(\frac{1}{2}\max \bigg\{\| YY^{*}+Z^{*}Z\|, \| Y^{*}Y+ZZ^{*}\|\bigg\}\bigg)^{\frac{1}{2}},\]
	\[w\left(\begin{array}{cc}
    X&Y \\
    Z&W
	\end{array}\right)\geq \max \bigg\{ w(X), w(W), \bigg(\frac{1}{4}\max \bigg\{\| YY^{*}+Z^{*}Z\|, \| Y^{*}Y+ZZ^{*}\|\bigg\}\bigg)^{\frac{1}{2}} \bigg\}.\]
\end{theorem}
\begin{proof}
The proof  follows easily from  Theorem \ref{theorem:upperlower1} and Lemma \ref{lemma:lemma3}.
\end{proof}

\begin{theorem}\label{theorem:estimate2}
Let $X, Y, Z, W \in \mathcal{B}(\mathcal{H})$. Then
\begin{eqnarray*}
w\left(\begin{array}{cc}
    X&Y \\
    Z&W
	\end{array}\right) &\leq & \max \bigg\{ w(X), w(W) \bigg \} + \bigg(\frac{1}{8}\max \big\{\alpha, \beta \big\}\bigg)^{\frac{1}{4}},\\ 
	w\left(\begin{array}{cc}
    X&Y \\
    Z&W
	\end{array}\right) &\geq & \max \bigg\{ w(X), w(W), \bigg(\frac{1}{16}\max \bigg\{\| A_0\|, \| B_0\|\bigg\}\bigg)^{\frac{1}{4}} \bigg\},
	\end{eqnarray*} 
	where 
	$$\alpha=\| YY^{*}+Z^{*}Z\|^2+4w^2(YZ),~~ \beta=\| Y^{*}Y+ZZ^{*}\|^2+4w^2(ZY),$$
	$$ A_0=\left(YY^{*}+Z^{*}Z\right)^2+4\left(\textit{Re}(YZ)\right)^2,~~B_0=\left(Y^{*}Y+ZZ^{*}\right)^2+4\left(\textit{Re}(ZY)\right)^2.$$
\end{theorem}
\begin{proof}
The proof  follows easily from  Theorem \ref{theorem:upperlower2} and Lemma \ref{lemma:lemma3}.
\end{proof}

\begin{remark}
We would like to remark that the first inequality of both Theorem \ref{theorem:estimate1} and Theorem \ref{theorem:estimate2} are valid even if we consider $X\in \mathcal{B}(\mathcal{H}_1), Y\in \mathcal{B}(\mathcal{H}_2,\mathcal{H}_1),Z\in \mathcal{B}(\mathcal{H}_1,\mathcal{H}_2), W\in \mathcal{B}(\mathcal{H}_2),$ where $\mathcal{H}_1,\mathcal{H}_2$ are Hilbert spaces.
\end{remark}

\section{On the bounds  for the zeros of a polynomial}

\noindent We begin this section with  some well-known bounds for the zeros of the polynomial $p(z).$  Let $\lambda$ be a zero of the polynomial  $p(z)$. Then \\

\noindent Cauchy \cite{CM}	proved that
 \begin{eqnarray}\label{zero1}
		    |\lambda| &\leq& 1+\max \left\{ |a_0|, |a_1|, \ldots, |a_{n-1}|\right\}.
\end{eqnarray}
Linden \cite{L} proved that 
\begin{eqnarray}\label{zero8}
   |\lambda| &\leq& \frac{|a_{n-1}|}{n}+\bigg[ \frac{n-1}{n} (n-1+\sum^{n-1}_{j=0}|a_j|^2-\frac{|a_{n-1}|^2}{n})\bigg]^{\frac{1}{2}}.
\end{eqnarray}
Kittaneh \cite{K} proved that 
\begin{eqnarray}\label{zero7}
   |\lambda| &\leq& \frac{1}{2}\left[|a_{n-1}|+1+\sqrt{(|a_{n-1}|-1)^2+4\sqrt{\sum^{n-2}_{j=0}|a_j|^2}}\right].
\end{eqnarray}
Abu-Omar and Kittaneh \cite{A} proved that 
\begin{eqnarray}\label{zero2}
|\lambda| &\leq& \frac{1}{2}\left[\beta+\sqrt{(\frac{1}{2}(|a_{n-1}|+\alpha)-\cos\frac{\pi}{n+1})^2+4\alpha' }\right],
\end{eqnarray}	
 where $\alpha=\sqrt{\sum_{j=0}^{n-1}|a_j|^2}$, $\alpha'=\sqrt{\sum_{j=0}^{n-2}|a_j|^2}$ and $\beta=\frac{1}{2}(|a_{n-1}|+\alpha)+\cos\frac{\pi}{n+1}$.\\
 Fujii and Kubo \cite{FK1} proved that 
	\begin{eqnarray}\label{zero3}
			|\lambda|&\leq& \cos\frac{\pi}{n+1}+\frac{1}{2}\left[\left(\sum_{j=0}^{n-1}|a_j|^2\right)^{\frac{1}{2}}+|a_{n-1}|\right].
	\end{eqnarray}			
 Alpin et al. \cite{YML} proved that
\begin{eqnarray}\label{zero4}
	     |\lambda| &\leq&  \max_{1\leq k \leq n}\left[(1+|a_{n-1}|)(1+|a_{n-2}|)\ldots(1+|a_{n-k}|)\right]^{\frac{1}{k}}.
\end{eqnarray}			
M. Al-Dolat et al. \cite{MKMF} proved that
\begin{eqnarray}\label{zero5}
        |\lambda| &\leq& \max \left\{ w(A), \cos\frac{\pi}{n+1}\right\} +\frac{1}{2} \left(1+\sqrt{{\sum_{j=0}^{n-3}|a_j|^2}} \right), 
\end{eqnarray}					
				where
$A=\left(\begin{array}{cc}
    -a_{n-1}&-a_{n-2} \\
    1&0
 \end{array}\right).$\\
 Bhunia et al. \cite{BBP} proved that 
\begin{eqnarray}\label{zero6}
   |\lambda| &\leq& \left( \frac{1}{2}w^2(C^2)+\frac{1}{4}\left\|(C^*C)^2+(CC^*)^2\right\|\right)^{\frac{1}{4}},
\end{eqnarray}			
where $C=C(p)$.\\
Carmichael and Mason \cite{CM} proved that 
\begin{eqnarray}\label{zero9}
		 |\lambda|\leq \left(1+|a_0|^2+|a_1|^2+\ldots+|a_{n-1}|^2\right)^{\frac{1}{2}}.
\end{eqnarray}

 As an application of the results obtained in section \ref{sec1}, we obtain some new bounds for the zeros of the polynomial $ p(z).$ To do so we need the following lemma.
\begin {lemma} [\cite{A}]\label{lemma:lemma4}
Let $D=\left(\begin{array}{ccccccc}
    0&0&.&.&.&0&0\\
		1&0&.&.&.&0&0\\
		0&1&.&.&.&0&0\\
    .& & & & & &\\
    .& & & & & &\\
    .& & & & & &\\
    0&0&.&.&.& 1&0
   \end{array}\right)_{n,n}.$  Then  $w(D)=\cos\frac{\pi}{n+1}$.
\end{lemma}
We now prove the following theorem.
\begin{theorem} \label{theorem:zero1}
Let $\lambda$ be any zero of $p(z).$ Then
\[|\lambda|\leq \max\left\{|a_{n-1}|,\cos\frac{\pi}{n}\right\}+\sqrt{\frac{1}{2}\left(1+\sum^n_{j=2}|a_{n-j}|^2\right)}.\]
\end{theorem}

\begin{proof}
Let $C(p)=\left(\begin{array}{cc}
   A&B \\
   C&D
 \end{array}\right),$\\ where $A=(-a_{n-1})_{1,1}$,  $B=(-a_{n-2} ~ -a_{n-3}~~ \ldots~~ -a_{1} ~ -a_{0})_{1,n-1}$,\\  $C^*=(1 ~0 ~\ldots ~0~ 0)_{1,n-1} $ and $D=\left(\begin{array}{ccccccc}
    0&0&.&.&.&0&0\\
		1&0&.&.&.&0&0\\
		0&1&.&.&.&0&0\\
    .& & & & & &\\
    .& & & & & &\\
    .& & & & & &\\
    0&0&.&.&.& 1&0
   \end{array}\right)_{n-1,n-1}.$\\
Therefore, using Lemma \ref{lemma:lemma4} and Theorem \ref{theorem:estimate1}, we get
\begin{eqnarray*}
w(C(p))&\leq& \max\bigg\{|a_{n-1}|,\cos\frac{\pi}{n}\bigg\}+\sqrt{\frac{1}{2}\max\bigg\{ \|B^{*}B+CC^{*}\|,\|BB^{*}+C^{*}C\|\bigg\}}\\
       &\leq& \max\bigg\{|a_{n-1}|,\cos\frac{\pi}{n}\bigg\}+\sqrt{\frac{1}{2}\bigg(\|B\|^2+\|C\|^2\bigg)}.
\end{eqnarray*}	
Thus,	\[|\lambda|\leq \max\bigg\{|a_{n-1}|,\cos\frac{\pi}{n}\bigg\}+\sqrt{\frac{1}{2}\bigg(1+\sum^n_{j=2}|a_{n-j}|^2\bigg)}.\]
This completes the proof of the theorem.	
\end{proof}

Similarly using Theorem \ref{theorem:estimate2}, we can prove the following theorem.
\begin{theorem}\label{theorem:zero2}
Let $\lambda$ be any zero of $p(z).$ Then
\[|\lambda|\leq \max\bigg\{|a_{n-1}|,\cos\frac{\pi}{n}\bigg\}+\left[\frac{1}{8}\bigg(1+\sum^n_{j=2}|a_{n-j}|^2\bigg)^2+\frac{1}{2}\sum^n_{j=2}|a_{n-j}|^2\right]^{\frac{1}{4}}.\]
\end{theorem}

We next illustrate with a numerical example  that the  bounds obtained  in Theorem \ref{theorem:zero1} and Theorem \ref{theorem:zero2} are stronger than the existing bounds. 
	\begin{example}
Consider a polynomial $p(z)=z^5+z^4+z^3+z^2+z+3$. Then the upper bounds for the zeros of this polynomial $p(z)$ estimated by different methods are as shown in the following table.
	
	\begin{center}
\begin{tabular}{ |c|c| } 
\hline
Abu-Omar and Kittaneh \cite{A} & 3.579 \\
 \hline
M. Al-Dolat et. al. \cite{MKMF} &  3.776\\ 
\hline
Cauchy \cite{CM} & 4.000  \\ 
 \hline
Lindon \cite{L} &   3.866\\ 
 \hline
Carmichael and Mason \cite{CM}& 3.741 \\
\hline
 \end{tabular}
\end{center}
But if $\lambda$ is a zero of this polynomial $p(z)$, then  Theorem \ref{theorem:zero1} gives $|\lambda| \leq 3.549$ and Theorem \ref{theorem:zero2} gives $|\lambda| \leq 3.292,$ which are stronger than all the estimations mentioned above.
\end{example}

Using Theorem \ref{theorem:estimate1}, we obtain another bound for the zeros of the polynomial $p(z).$

\begin{theorem}\label{theorem:zero3}
Let $\lambda$ be any zero of $p(z)$. Then
\begin{eqnarray*} 
 |\lambda|&\leq& \left|\frac{a_{n-1}}{n}\right|+\cos\frac{\pi}{n}+\sqrt{\frac{1}{2}(1+\alpha)},  
\end{eqnarray*}
where 
\begin{eqnarray*}
\alpha_{r}&=&\sum^{n}_{k=r} {^k}C_r\big(-\frac{a_{n-1}}{n}\big)^{k-r} a_{k},  ~~ r=0,1,\ldots, n-2, a_n=1, {^0}C_0=1, \\
\alpha&=&\sum^{n}_{j=2}|\alpha_{n-j}|^2.
\end{eqnarray*}
\end{theorem}
\begin{proof}
First we put $z=\eta -\frac{a_{n-1}}{n}$ in the polynomial $p(z).$ Then, we get a polynomial $$q(\eta)=\eta^n+\alpha_{n-2}\eta^{n-2}+\alpha_{n-3}\eta^{n-3}+\ldots +\alpha_{1}\eta+\alpha_{0},$$  where $\alpha_{r}=\sum^{n}_{k=r} {^k}C_r\big(-\frac{a_{n-1}}{n}\big)^{k-r} a_{k},  ~~ r=0,1,\ldots, n-2$, $a_n=1$ and $^{0}C_0=1$.

\noindent Therefore, the Frobenius companion matrix for the polynomial $q(\eta)$ is $C(q)=\left(\begin{array}{cc}
    A&B \\
    C&D
	\end{array}\right),$ where  $A=(0)_{1,1}$,  $B=(-\alpha_{n-2} ~~ -\alpha_{n-3}~~ \ldots~~ -\alpha_{1} ~~ -\alpha_{0})_{1,n-1}$,  $C^*=(1 ~~0 ~~\ldots ~~0~~ 0)_{1,n-1} $ and  $D=\left(\begin{array}{ccccccc}
    0&0&.&.&.&0&0\\
		1&0&.&.&.&0&0\\
		0&1&.&.&.&0&0\\
    .& & & & & &\\
    .& & & & & &\\
    .& & & & & &\\
    0&0&.&.&.& 1&0
   \end{array}\right)_{n-1,n-1}.$\\
Using Lemma \ref{lemma:lemma4} and Theorem \ref{theorem:estimate1}, we get
\begin{eqnarray*}
w(C(q))&\leq& \cos\frac{\pi}{n}+\big(\frac{1}{2}\max\big\{ \|B^{*}B+CC^{*}\|,\|BB^{*}+C^{*}C\|\big\}\big)^{\frac{1}{2}}\\
       &\leq& \cos\frac{\pi}{n}+\big(\frac{1}{2}\big(\|B\|^2+\|C\|^2\big)\big)^{\frac{1}{2}}.
\end{eqnarray*}	
Therefore, if $\eta$ is any zero of the polynomial $q(\eta)$ then
\[|\eta|\leq \cos\frac{\pi}{n}+\bigg(\frac{1}{2}\big(\|B\|^2+\|C\|^2\big)\bigg)^{\frac{1}{2}}.\]
Thus, we get \[|\lambda|\leq \left|\frac{a_{n-1}}{n}\right|+\cos\frac{\pi}{n}+\sqrt{\frac{1}{2}(1+\alpha)}.\]
This completes the proof of the theorem.
\end{proof}

Similarly using Theorem \ref{theorem:estimate2}, we obtain another bound  for the zeros of the  polynomial $ p(z).$

\begin{theorem}\label{theorem:zero4}
Let $\lambda$ be any zero of $p(z)$. Then

\begin{eqnarray*} 
 |\lambda|&\leq& \left|\frac{a_{n-1}}{n}\right|+\cos\frac{\pi}{n}+\bigg[\frac{1}{8}(1+\alpha)^2+\frac{1}{2}\alpha\bigg]^{\frac{1}{4}},  
\end{eqnarray*}
where 
\begin{eqnarray*}
\alpha_{r}&=&\sum^{n}_{k=r} {^k}C_r\big(-\frac{a_{n-1}}{n}\big)^{k-r} a_{k},  ~~ r=0,1,\ldots, n-2, a_n=1, {^0}C_0=1, \\
\alpha&=&\sum^{n}_{j=2}|\alpha_{n-j}|^2.
\end{eqnarray*}
\end{theorem}

We again  provide a numerical example to show that the  bounds obtained  in Theorem \ref{theorem:zero3} and Theorem \ref{theorem:zero4} are stronger than the existing bounds. 
	\begin{example}
	Consider a  polynomial $p(z)=z^5+2z^4+2z^3+z^2+2z+2$. Then the upper bounds  for  the zeros of this polynomial $p(z)$ estimated by different methods are as shown in the following table.
	\end{example}
	\begin{center}
\begin{tabular}{ |c|c| } 
\hline
Abu-Omar and Kittaneh \cite{A} & 4.157 \\
\hline
Alpin et al. \cite{YML} & 3.000 \\
\hline
Bhunia et al. \cite{BBP} & 3.183 \\
\hline
Fujii and Kubo \cite{FK1} & 3.927  \\
	\hline 
Cauchy \cite{CM} &  3.000 \\ 
 \hline
Lindon \cite{L} &  4.419 \\ 
 \hline
Kittaneh \cite{K}& 3.463 \\
\hline
\end{tabular}
\end{center}
But if $\lambda$ is a zero of this polynomial $p(z),$ then  Theorem \ref{theorem:zero3} gives $|\lambda| \leq 2.933$ and Theorem \ref{theorem:zero4} gives $|\lambda| \leq 2.829,$ which are stronger than all the bounds mentioned above.

\begin{acknowledgements}
We would like to thank the referee for his/her helpful suggestions.
Mr. Pintu Bhunia would like to thank UGC, Govt. of India for the financial support in the form of SRF. Prof. Kallol Paul would like to thank RUSA 2.0, Jadavpur University for the partial support.
\end{acknowledgements}

\bibliographystyle{amsplain}

\end{document}